\documentclass[12pt]{amsart}
\usepackage{amsmath,amsthm,amssymb,amsfonts}
\newtheorem{theorem}{Theorem}

\theoremstyle{definition}

\theoremstyle{remark}
\newtheorem{remark}[theorem]{Remark}

\begin{document}

    \title[Conformal scalar curvature rigidity]{Conformal scalar curvature rigidity on Riemannian manifolds}

    \author{Seongtag Kim}

    \address{ Department of Mathematics Education, Inha University, Incheon 22212, Korea and Department of Mathematics, Princeton University, NJ 08544, USA}
    \email{stkim@inha.ac.kr}



    \keywords{conformal metrics, Yamabe Problem, scalar
        curvature}


    \begin{abstract}
    Let  $(M, \bar g)$ be an $n$-dimensional complete Riemannian
    manifold.
In this paper, we considers the following conformal scalar
curvature rigidity problem: Given a compact smooth domain $\Omega$
with $\partial \Omega$,  can one find a conformal metric $g$ whose
scalar curvature $R[g]\ge R[\bar g]$ on $\Omega$ and the mean
curvature $H[g] \ge H[ \bar g]$ on $\partial \Omega$  with  $\bar
g = g$ on $\partial \Omega$?
        We prove that   $\bar g = g$  on some
        smooth domains in a general Riemannian manifold, which  is
        an extension of the previous results given by Qing and Yuan, and
        Hang and Wang.
    \end{abstract}

    \maketitle

    \section{Introduction}

 Let  $(M, \bar g)$ be an $n$-dimensional complete Riemannian manifold and $\Omega$ be a smooth domain in  $(M, g)$ with
 smooth boundary $\partial \Omega$. Denote $R  [\bar g]$  by the scalar curvature of $\bar g$
 and $H [\bar g]$   the mean curvature of $\partial \Omega$.
 In this paper, we considers the following problem: Given a compact smooth domain $\Omega$ with
$\partial \Omega$,  can one find a conformal metric $g$ whose
scalar curvature $R[g]\ge R[\bar g]$ on $\Omega$ and the mean
curvature $H[g] \ge H[ \bar g]$ on  $\partial \Omega$  with  $\bar
g = g$ on $\partial \Omega$?

For  the conformal metric $  g$ of the given metric $(M,
\bar{g})$, scalar curvature $R[g]$ and mean curvature
    $H[g]$ on the boundary of the domain change in the following ways:
    \begin{equation}\label{eq001}\begin{split}
    g &=   e^{2u}\bar g,  \quad
   R[g] = e^{-2u} \left( R[\bar{g}] - 2 \Delta u \right)  \text {
   and} \\
H[g]& = e^{-u} [H[\bar g] + \partial_\nu u ] \, \text{ when} \,
n=2,
\end{split}
\end{equation}
\begin{equation}\label{eq002}\begin{split}
g &= u^\frac 4{n-2}\bar g, R[g] = u^{-\frac{n+2}{n-2}} \left(
R[\bar{g}] u - \frac {4(n-1)}{n-2}\Delta u \right) \text { and} \\
H[g]&
 =u^{-\frac n {n-2}} \Big[H[\bar g] + \frac
{2}{n-2}\partial_\nu u \Big] \text { when} \,
 n\ge3.
 \end{split} \end{equation}
 Therefore the given condition $R[g]\ge R[\bar g]$ on $\Omega$ is
 equivalent to
\begin{equation}\label{eqq1}\begin{split}
 R[\bar{g}]e^{2u} \le  \left( R[\bar{g}] - 2 \Delta u \right) \, \text{ when} \,
 n=2 \, \text{ and}  \\
R[\bar{g}] u^{\frac{n+2}{n-2}} \le \left( R[\bar{g}] u - \frac
{4(n-1)}{n-2} \Delta u \right) \, \text{ when} \,   n\ge3.
 \end{split}  \end{equation}
The condition $H[g] \ge H[ \bar g]$  with $\bar g = g$ on
$\partial \Omega$  is
 equivalent to
\begin{equation}\label{eqq2}\begin{split}
\partial_\nu u \ge 0.
 \end{split} \end{equation}

This problem is a conformal version of Min-Oo's conjecture.
Uniqueness  and non-uniqueness of conformal metric with prescribed
scalar curvature on $\Omega$ with the mean curvature condition on
$\partial \Omega$ was studied by Escobar. He proved that on the
annulus $A_{a,b}=\{x \in R^n| 0<a< |x|<b\}$ with the Euclidean
metric $\delta_{ij}$ admits a conformal metrics of the form
$u^{4/(n-2)}(|x|)\delta_{ij}  $ that there exist several metrics
with the same constant scalar curvature and the same constant mean
curvature on the boundary if $\frac b a$ is big enough
\cite{es90}. Therefore we can not expect to get a conformal scalar
curvature rigidity to our question in general. However, if $\frac
b a$ is small and $\bar g = g$ on $|x|=a$ and $|x|= b$, our result
implies that $\bar g = g$ on $A_{a,b}$.
 Hang and Wang obtained the
conformal deformations rigidity of metrics on the hemisphere
$D=S^+$ in the standard sphere $(S^n,\bar g)$ where $\bar g=g_0$.
 They proved that (\ref{eqq1}) and $\bar g = g$ on
$\partial D$ imply $\bar g = g$ on $D$ \cite{H-W1}. This result
was recently generalized to the domains in the vacuum static
spaces by Qing and Yuan.
    \begin{theorem}\label{QYthm}\cite[Theorem 5.1]{Q-Y2}
    Let $(M^n,\bar{g}, f)$ be a complete $n$-dimensional static space with $R_{\bar{g}} > 0$ ($n \geq 2$). Assume $\Omega^+\equiv\{p\in M| f(p)>0\}$
    is a pre-compact subset in $M$.
    Then, if a metric $g \in [\bar{g}]$ on $M$ satisfies that
    \begin{itemize}
        \item $R[g] \geq R[\bar{g}]$ in $\Omega^+$,
        \item $g$ and $\bar{g}$ induced the same metric on $\partial \Omega^+$, and
        \item $H[g] = H[\bar{g}]$ on $\partial \Omega^+$,
    \end{itemize}
    then $g =\bar{g}$.
\end{theorem} \noindent
For the proof of Theorem \ref{QYthm}, they used the existence of
the lapse functions $f$ with the following  properties:
\begin{equation}\label{lapse-equ}
- \Delta f - \frac {R[\bar g]}{n-1} f = 0 \text{ and } f > 0
\text{ in $\Omega^+$} \end{equation} and
\begin{equation}\label{bdy-cond}
\nabla f\neq 0  \, \text { at} \,  \partial\Omega^+ = \{x\in M^n|
f(x) = 0\},
\end{equation}
where  $\Omega^+$ is the maximal subset where the conformal
rigidity holds. The existence of lapse function comes from the
vacuum static space $(M,g)$ and  (\ref{bdy-cond}) holds for $f$
(see: \cite[Theorem 1]{F-M}).  In this paper, we extend the
previous conformal scalar curvature rigidity results to the
domains in a   general Riemannian manifold with the conformal
invariant.

\section{Conformal Rigidity of Scalar curvature}\label{Conformal_Rigidity_Section}

Let $(M, g)$ be a complete Riemannian manifold  of dimension $n
\ge 3$ with scalar curvature $R[ g]$. The Sobolev constant  $Q(M,
g)$ of  $(M, g)$ and $Q(\Omega, g) $ of a smooth domain
$\Omega\subset (M, g) $ are defined by
$$Q(M, g)\equiv \inf_{0\neq u\in C^{\infty}_0 (M)}  { { {     \int_M |\nabla u|^2+{{(n-2)}\over {4(n-1)} } R_g u^2 dV_g}
                  \over {\left( \int_M |u|^{   {2n}/(n-2)} dV_g
                  \right)^{(n-2)/n}}} }  $$ and

$$Q(\Omega, g)\equiv \inf_{0\neq u\in C^{\infty}_0 (\Omega)}
{ { {     \int_M |\nabla u|^2+{{(n-2)}\over {4(n-1)} } R_g u^2 dV_g}
                  \over {\left( \int_M |u|^{   {2n}/(n-2)} dV_g
                  \right)^{(n-2)/n}}} }. $$

Note that $Q(M, g)$ and $Q(\Omega, g)$ are conformal invariant and
$Q(\Omega, g) \ge Q(M, g)$. There are domains in a complete
Riemannian manifolds  with positive Sobolev constant.  For
example, any simply connected domain in a complete locally
conformally flat manifold has positive  Sobolev constant
\cite{sy}. It is known that for any smooth domain $\Omega \subset
(M, g)$, $Q(\Omega, g) \le Q(S^n, g_0)$ where $(S^n, g_0)$ is the
standard sphere (see \cite {au}). Using $Q(\Omega, g)$ and $Q(M,
g)$, we obtain conformal rigidity phenomena of scalar curvature.
Let $R^+(x)=\sup(0, R[\bar g](x))$.
\begin{theorem}\label{Kim-1}
    Let $(M^n,\bar{g})$ be a complete  Riemannian $n$-manifold with
    scalar curvature $R[\bar g]$ ($n \geq 3$) and
 $\Omega$ be a smooth domain in  $(M^n,\bar{g})$ with positive $Q(\Omega, \bar
 g)>0$. Assume that
 $ {\frac {(n+2)} {4(n-1)} } \Big[ \int_{\Omega }
                     |R^+| _{\bar g} ^{\frac n 2}  dV_{\bar g}\Big]^{\frac 2n }< Q(\Omega, \bar g)$.
    Then, if a conformal metric $g \in [\bar{g}]$ on $M$ satisfies that
    \begin{itemize}
        \item $R[g] \geq R[\bar{g}]$ in $\Omega$,
        \item $g$ and $\bar{g}$ induced the same metric on $\partial \Omega$, and
        \item $H[g] \ge H[\bar{g}]$ on $\partial \Omega$,
    \end{itemize}
    then $g =\bar{g}$.
\end{theorem}
\begin{proof}
Since $ R[\bar g] \le R[g]$ we have (\ref{eqq1}).  Take $v=u-1$,
$\Omega_1=\{x \in \Omega | u(x) < 1\}$  and $\Omega_2=\{x \in
\Omega | u(x) \ge 1\}$. We shall show that $\Omega_1=\phi$. If
then, $u>1$ on $\Omega$ by the maximum principle. Since the mean
curvature at the boundary is increasing, $\partial_\nu u\ge 0$ on
$\partial \Omega$. However, this contradicts to the strong maximum
principle since $u>1$ on $\Omega$ and $u=1$ on the $\partial
\Omega$  (see \cite{Q-Y2}). To show that $\Omega_1=\phi$, we let
\begin{equation}\label{AAA} A(x)=\frac{\frac {n-2}{4(n-1)}
R[\bar{g}] u(x) \left(u(x)^{\frac{4}{n-2}} - 1\right)}{u(x)-1}.
\end{equation} The given conditions imply that
\begin{equation}\label{q11}
- \Delta v - A(x) v \geq 0
\end{equation}
 on $\Omega$, $v=0$ on  $\partial \Omega$ and $\partial_\nu v=0$ on $\partial \Omega$.
Note that $A(x)\le \frac 1 {n-1} R^+[\bar{g}]$ on $\Omega_1$.
Multiplying $v$ on  (\ref{q11}) on $\Omega_1$,
\begin{equation}\label{q12}
\int_{\Omega_1} |\nabla v|^2- A(x) v^2 dV_{\bar g} \le 0.
\end{equation}
We may  consider $v$ as a function defined on $\Omega$ by extending the domain.
By using the Sobolev constant $Q(\Omega, g)$ of $\Omega$,
\begin{equation}\label{qq11} \begin{split}
 Q(\Omega, \bar g)  {\left( \int_\Omega |v|^{   \frac{2n}{n-2}}
dV_{\bar g}
                    \right)^ {\frac{n-2}n}    }
                   &   \le  \int_{\Omega _1} |\nabla v|^2+{{(n-2)}\over {4(n-1)} } R_{\bar g} v^2 dV_{\bar g} \\
                  & \le       \int_{\Omega _1} \Big( A(x)+{{(n-2)}\over {4(n-1)} } R_{\bar g} \Big)v^2 dV_{\bar g}\\
                  & \le  \Big( {\frac 1 {n-1}+\frac{(n-2)} {4(n-1)} }\Big)     \int_{\Omega _1}  R^+_{\bar g} v^2 dV_{\bar g}\\
                   & \le    {\frac {(n+2)} {4(n-1)} }   \int_{\Omega _1} |R^+| _{\bar g} v^2 dV_{\bar g}\\
                    & \le         {\frac {(n+2)} {4(n-1)} }  \Big[ \int_{\Omega _1}
                      |R^+| _{\bar g} ^{\frac n 2}  dV_{\bar g}\Big] ^{\frac 2n }
                     \Big[ \int_{\Omega _1} |v|^{\frac {2n}{n-2}} dV_{\bar g}\Big]^{\frac{n-2}n},
                     \end{split} \end{equation}
                     where (\ref{q12}) is used.
  Therefore    $ {\frac {(n+2)} {4(n-1)} } \Big[ \int_{\Omega _1}
                      |R^+| _{\bar g} ^{\frac n 2}  dV_{\bar g}\Big]^{\frac 2n }< Q(\Omega, g)$
                     implies $v\equiv0$ on ${\Omega _1}$.
\end{proof}

\begin{remark} Let  $\Omega$ be
 a simply connected domain  in a locally conformally flat
manifold $(M, \bar g)$, then  $Q(\Omega, g)= Q(S^n, g_0)$. If $
{\frac {(n+2)} {4(n-1)} } \Big[ \int_{\Omega }
                      |R^+| _{\bar g} ^{\frac n 2}  dV_{\bar g}\Big]^{\frac 2n }< Q(S^n, g_0)$,
                      then Theorem \ref{Kim-1} holds
                      for $\Omega$. If $(M, g)$ is locally conformally flat with constant positive
scalar curvature, then the rigidity holds for $\Omega$ with
sufficiently small $|\Omega|$.
\end{remark}
\begin{remark}
When $(M^n,\bar g)$ is a compact Einstein manifold with positive
scalar curvature, $Q(M^n,\bar g)={{(n-2)}\over {4(n-1)} } R[\bar
g] [Vol(M,g)]^{\frac 2n}$ \cite[page 48]{He}. Since $Q(M, \bar
g)\le Q(\Omega, \bar g)$,  Theorem \ref{Kim-1} holds for any
smooth domain $\Omega$ with $|\Omega|< [\frac{n-2}{n+2}]^{\frac
n2}[Vol(M,g)]$. \end{remark} For $n \geq 2$, let $(M^n,\bar{g})$
be a Riemannian space with positive scalar curvature
 $R[\bar{g}] > 0$. Next  we prove that for each given point $p$ there exist
 domain $D \ni p $ in a general manifold, on which conformal scalar curvature rigidity
 holds by  applying the techniques of \cite{Q-Y2}.
For a domain $ D \subset (M, \bar{g}) $, we let $$R(\bar{g}, D)=
\sup\limits_{\{p\in D\}}R[\bar{g}](p)$$ and  $\lambda_1(D)$ be the
1-st nonzero eigenvalue of domain $D$ with Dirichlet condition
with respect to the metric $\bar{g}$, i.e.,
 $$\lambda_1(D)=\inf_{u\in H}\frac{\int_D |\nabla u|^2 dV_{\bar{g}}} {\int_D u^2 dV_{\bar{g}}}$$
 where $u\in H={H}_0^{1,2}({D})$. It is known that for a given point $p\in (M,
 g)$, we can find $D\ni p$ with sufficiently large $\lambda_1(D)$.
\begin{theorem}
Let $(M^n,\bar{g})$ be a complete $n$-dimensional Riemannian space
with $R_{\bar{g}} > 0$ ($n \geq 2$). Assume $D$ is a smooth
pre-compact subset in $M$ with $\lambda_1(D)> \frac{R(\bar{g},
D)}{n-1}  $. Then, if a metric $g \in [\bar{g}]$ on $M$ satisfies
that
\begin{itemize}\label{eigen}
\item $R[g] \geq R[\bar{g}]$ in $D$, \item $g$ and $\bar{g}$
induced the same metric on $\partial D$, and \item $H[g] =
H[\bar{g}]$ on $\partial D$,
\end{itemize}
then $g =\bar{g}$.
\end{theorem}
\begin{proof}
  Let $g=u^{\frac 4{n-2}} \bar{g}$.
    To prove the rigidity on  the domains in a general Riemannian
    manifold,
    we construct a positive smooth function  on a suitable domain $D$
    with the  properties similar to (\ref{lapse-equ}, \ref{bdy-cond}).
    For this, we take any smooth domain $D \subset \subset M$ with  $\frac {R(\bar{g},D)}{n-1}  <\lambda_1(D) $
    and  the eigenfunction $u_1$ with  $\Delta u_1+\lambda u_1=0$ on $D$.
Take $v=u-1$.   Since $u_1$ is positive on $D$, we can express $
v= u_1 \beta$ with some function $\beta$ on $D$. From (\ref{q11}),
 \begin{equation}\label{eq12}\begin{split}
    0\ge & \Delta v+ A(x) v \\ = & \triangle (u_1 \beta)+ A(x) u_1 \beta \\
    = & u_1 \triangle \beta+\Delta u_1 \beta+\nabla u_1 \cdot \nabla
\beta +A(x) u_1 \beta  \\
   = & u_1\triangle \beta-\lambda  u_1 \beta+ \nabla u_1 \cdot \nabla \beta +A(x) u_1
      \beta. \\
 \end{split} \end{equation}
 Since $u_1>0$ on $D$,
    \begin{equation}\label{eq13}
    \begin{split}
     0\ge &\triangle \beta -\lambda   \beta+\frac{\nabla u_1}{u_1}\cdot\nabla\beta +A(x)  \beta \\
     0\ge &\triangle \beta+(A(x)-\lambda)   \beta+\frac{\nabla u_1}{u_1}\cdot\nabla\beta.  \\
      \end{split} \end{equation} Using L'hospital's rule,
 $\beta= 0$ on $\partial D$. If
     there exists a minimum point $p\in D$ with $\beta(p)<0$ and $\nabla\beta(p)=0$,
    then  $\triangle \beta(p)<0$ since   $A(x)\le
    \lambda_1(D)$ if $v\le 0$. This contradicts to the Maximum
    principle (see \cite{C-L}).
Therefore $v\ge 0$ on $D$. Then by the Maximum principle again, it
does not satisfy the boundary condition $\partial_\nu u = 0 $ on
$\partial D$ if $v$ is not identically zero on $D$.
\end{proof}

Note that on a standard hemisphere $S^+$, $R[\bar{g}]=n (n-1)$
$\lambda_1(S^+)=n$ (see \cite{Re}), which provides maximal domain.
Any smaller domain than the hemisphere $D \subset \subset S^+$
satisfies $\lambda_1(D)>n$, on which  Theorem \ref{eigen} holds.

\bibliographystyle{amsplain}

\begin{thebibliography}{10}
\bibitem{au}
T. Aubin, Some nonlinear problems in Riemannian geometry.
Springer-Verlag,
 Berlin, 1998.

\bibitem{C-L} W. Chen and C. Li, \textit{Methods on Nonlinear Elliptic Equations}, AIMS (2010)

\bibitem{es90}
J. Escobar, \textit{ Uniqueness theorems on conformal deformation of metrics, Sobolev inequalities, and an eigenvalue estimate} Vol 43, No7,  Comm. Pure Appl. Math.  (1990),   857–-883.

\bibitem{F-M} A. Fischer and J. Marsden, \textit{Deformations of the scalar curvature}, Vol.42, No.3 Duke Mathematical Journal (1975) 519 - 547.

\bibitem{H-W1} F. Hang and X. Wang, \textit{Rigidity and non-rigidity results on the sphere}, Comm. Anal. Geom. \textbf{14}, (2006) 91 - 106.
\bibitem{He} E. Hebey, Sobolev spaces on Riemannian manifolds. Springer-Verlag, Berlin, 1996.
\bibitem{Q-Y2} J. Qing and W. Yuan \textit{On scalar curvature rigidity of Vacuum Static Spaces},
Math. Ann.  \textbf{365}, (2016) 1257--1277.

\bibitem{Re}
R. Reilly,  \textit{ Applications of the Hessian operator in a
Riemannian manifold}, Ind. Univ. Math. J. \textbf{ 26}, (1977)
459-472.

\bibitem{sy}
R. Schoen, S. T. Yau, \textit{ Conformally flat manifolds,
Kleinian groups and scalar curvature}, Invent. Math.  \textbf{92},
(1988), 47--71.

\end{thebibliography}

\end{document}